\documentclass[12pt]{amsart}
\usepackage{amsmath,amssymb,amsbsy,amsfonts,latexsym,amsopn,amstext,cite,
                                               amsxtra,euscript,amscd,bm,mathabx}
\usepackage{url}

\usepackage[colorlinks,linkcolor=blue,anchorcolor=blue,citecolor=blue,backref=page]{hyperref}
\usepackage{color}
\usepackage{graphics,epsfig}
\usepackage{graphicx}
\usepackage{float}
\usepackage{epstopdf}
\hypersetup{breaklinks=true}
 
\usepackage[english]{babel}
\usepackage{mathtools}
\usepackage{todonotes}
\usepackage{url}
\usepackage[colorlinks,linkcolor=blue,anchorcolor=blue,citecolor=blue,backref=page]{hyperref}

\usepackage[norefs,nocites]{refcheck}

\begin{document}

\newtheorem{theorem}{Theorem}
\newtheorem{lemma}[theorem]{Lemma}
\newtheorem{claim}[theorem]{Claim}
\newtheorem{cor}[theorem]{Corollary}
\newtheorem{conj}[theorem]{Conjecture}
\newtheorem{prop}[theorem]{Proposition}
\newtheorem{definition}[theorem]{Definition}
\newtheorem{question}[theorem]{Question}
\newtheorem{example}[theorem]{Example}
\newcommand{\hh}{{{\mathrm h}}}
\newtheorem{remark}[theorem]{Remark}

\numberwithin{equation}{section}
\numberwithin{theorem}{section}
\numberwithin{table}{section}
\numberwithin{figure}{section}

\def\sssum{\mathop{\sum\!\sum\!\sum}}
\def\ssum{\mathop{\sum\ldots \sum}}
\def\iint{\mathop{\int\ldots \int}}

\def\MSB{{\mathrm{MSB}}}

\newcommand{\diam}{\operatorname{diam}}

\def\squareforqed{\hbox{\rlap{$\sqcap$}$\sqcup$}}
\def\qed{\ifmmode\squareforqed\else{\unskip\nobreak\hfil
\penalty50\hskip1em \nobreak\hfil\squareforqed
\parfillskip=0pt\finalhyphendemerits=0\endgraf}\fi}

\newfont{\teneufm}{eufm10}
\newfont{\seveneufm}{eufm7}
\newfont{\fiveeufm}{eufm5}
%
%
\newfam\eufmfam
     \textfont\eufmfam=\teneufm
\scriptfont\eufmfam=\seveneufm
     \scriptscriptfont\eufmfam=\fiveeufm
%
%
\def\frak#1{{\fam\eufmfam\relax#1}}

\newcommand{\bflambda}{{\boldsymbol{\lambda}}}
\newcommand{\bfmu}{{\boldsymbol{\mu}}}
\newcommand{\bfxi}{{\boldsymbol{\xi}}}
\newcommand{\bfrho}{{\boldsymbol{\rho}}}

\def\eps{\eta}

\def\fK{\mathfrak K}
\def\fT{\mathfrak{T}}

\def\fA{{\mathfrak A}}
\def\fB{{\mathfrak B}}
\def\fC{{\mathfrak C}}
\def\fM{{\mathfrak M}}
\def\fS{{\mathfrak  S}}
\def\fU{{\mathfrak U}}

\def\T {\mathsf {T}}
\def\Tor{\mathsf{T}_d}
\def\Tore{\widetilde{\mathrm{T}}_{d} }

\def\sM {\mathsf {M}}

\def\Rk{\cR_k}
\def\Kmnp{\cK_p(m,n)}
\def\Kmnq{\cK_q(m,n)}

\def \balpha{\bm{\alpha}}
\def \bbeta{\bm{\beta}}
\def \bgamma{\bm{\gamma}}
\def \bdelta{\bm{\delta}}
\def \blambda{\bm{\lambda}}
\def \bchi{\bm{\chi}}
\def \bphi{\bm{\varphi}}
\def \bpsi{\bm{\psi}}
\def \bnu{\bm{\nu}}
\def \bomega{\bm{\omega}}

\def \bxi{\bm{\xi}}

\def \bell{\bm{\ell}}

\def\eqref#1{(\ref{#1})}

\def\vec#1{\mathbf{#1}}

\def\fl#1{\left\lfloor#1\right\rfloor}
\def\rf#1{\left\lceil#1\right\rceil}

\newcommand{\abs}[1]{\left| #1 \right|}

\def\Zq{\mathbb{Z}_q}
\def\Zqx{\mathbb{Z}_q^*}
\def\Zd{\mathbb{Z}_d}
\def\Zdx{\mathbb{Z}_d^*}
\def\Zf{\mathbb{Z}_f}
\def\Zfx{\mathbb{Z}_f^*}
\def\Zp{\mathbb{Z}_p}
\def\Zpx{\mathbb{Z}_p^*}
\def\cM{\mathcal M}
\def\cE{\mathcal E}
\def\cH{\mathcal H}

\def\sfB{\mathsf {B}}
\def\sfC{\mathsf {C}}
\def\L{\mathsf {L}}
\def\FF{\mathsf {F}}

\def\cA{{\mathcal A}}
\def\cB{{\mathcal B}}
\def\cC{{\mathcal C}}
\def\cD{{\mathcal D}}
\def\cE{{\mathcal E}}
\def\cF{{\mathcal F}}
\def\cG{{\mathcal G}}
\def\cH{{\mathcal H}}
\def\cI{{\mathcal I}}
\def\cJ{{\mathcal J}}
\def\cK{{\mathcal K}}
\def\cL{{\mathcal L}}
\def\cM{{\mathcal M}}
\def\cN{{\mathcal N}}
\def\cO{{\mathcal O}}
\def\cP{{\mathcal P}}
\def\cQ{{\mathcal Q}}
\def\cR{{\mathcal R}}
\def\cS{{\mathcal S}}
\def\cT{{\mathcal T}}
\def\cU{{\mathcal U}}
\def\cV{{\mathcal V}}
\def\cW{{\mathcal W}}
\def\cX{{\mathcal X}}
\def\cY{{\mathcal Y}}
\def\cZ{{\mathcal Z}}
\newcommand{\rmod}[1]{\: \mbox{mod} \: #1}

\def\cg{{\mathcal g}}

\def\vr{\mathbf r}
\def\vx{\mathbf x}
\def\va{\mathbf a}
\def\vb{\mathbf b}
\def\vc{\mathbf c}
\def\vh{\mathbf h}
\def\vk{\mathbf k}
\def\vz{\mathbf z}
\def\vu{\mathbf u}
\def\vv{\mathbf v}

\def\e{{\mathbf{\,e}}}
\def\ep{{\mathbf{\,e}}_p}
\def\eq{{\mathbf{\,e}}_q}

\def\Tr{{\mathrm{Tr}}}
\def\Nm{{\mathrm{Nm}}}

 \def\SS{{\mathbf{S}}}

\def\lcm{{\mathrm{lcm}}}

 \def\0{{\mathbf{0}}}
 
 \def\tf{\widetilde f}
 
 \def\va {{\mathbf a}}
\def \vb {{\mathbf b}}
\def \vc {{\mathbf c}}
\def\vx{{\mathbf x}}
\def \vr {{\mathbf r}}
\def \vs {{\mathbf s}}
\def \vv {{\mathbf v}}
\def\vu{{\mathbf u}}
\def \vw{{\mathbf w}}
\def \vz {{\mathbf z}}
\def \ve {{\mathbf e}}
\def \vp {{\mathbf p}}

\def\ord {\mathrm{ord}}

\def\near#1{\left\lfloor #1\right\rceil}
\def\cvp#1{\left\lfloor{\vec{#1}}\right\rceil}
 
 \def \vol {\mathrm{Vol}}

\def\dist {{\mathrm{\,dist\,}}}

\def\({\left(}
\def\){\right)}
\def\l|{\left|}
\def\r|{\right|}
\def\fl#1{\left\lfloor#1\right\rfloor}
\def\rf#1{\left\lceil#1\right\rceil}
\def\sumstar#1{\mathop{\sum\vphantom|^{\!\!*}\,}_{#1}}

\def\mand{\qquad \mbox{and} \qquad}

 \def\dmod#1#2{\left|#1\right|_{#2}}

\def\flpk#1{{\left\lfloor#1\right\rfloor}_{p^k}}
\def\flp#1{{\left\lfloor#1\right\rfloor}_p}
\def\flm#1{{\left\lfloor#1\right\rfloor}_m}
\def\flq#1{{\left\lfloor#1\right\rfloor}_q}




\hyphenation{re-pub-lished}

\mathsurround=1pt

\def\bfdefault{b}

\def \F{{\mathbb F}}
\def \K{{\mathbb K}}
\def \N{{\mathbb N}}
\def \Z{{\mathbb Z}}
\def \Q{{\mathbb Q}}
\def \R{{\mathbb R}}
\def \C{{\mathbb C}}
\def\Fp{\F_p}
\def \fp{\Fp^*}

 \def \xbar{\overline x}

\title[Noisy polynomial interpolation]{
Noisy polynomial interpolation modulo prime powers}

\author[M. Karpinski]{Marek Karpinski}
\address{Department of Computer Science, Bonn University, 
 53113 Bonn, Germany} 
\email{marek@cs.uni-bonn.de}

\author[I. E. Shparlinski] {Igor E. Shparlinski} 
\address{Department of Pure Mathematics, University of New South Wales, 
Sydney, NSW 2052, Australia}
\email{igor.shparlinski@unsw.edu.au}

\date{\today}

\begin{abstract}
We consider  the {\it noisy polynomial interpolation problem\/}
of recovering an unknown $s$-sparse polynomial $f(X)$ over the  ring $\Z_{p^k}$
of residues modulo $p^k$, where $p$ is a small prime and $k$ is a large integer parameter,
from approximate values of the residues of $f(t) \in \Z_{p^k}$. 
Similar results are known for residues modulo a large prime $p$, however the case of 
prime power modulus $p^k$, with small $p$ and large $k$, is new and requires different techniques. We give a deterministic
polynomial time algorithm, which for almost given more than a half bits of $f(t)$ 
for sufficiently many randomly chosen points $t \in \Z_{p^k}^*$, recovers $f(X)$. 
\end{abstract}

 

\maketitle

\section{Introduction}


There is a long history and very extensive literature dedicated to algorithms on polynomials in finite fields, see, for example~\cite{vzGG}.
More recently, there was also increasing interest to algorithms for polynomials over residue rings, especially in residue rings modulo prime 
powers, see~\cite{CGRW,DMS1,DMS2,vzGH,  HJPW, KRRZ,  Sir} and references therein. Here we continue this directions and consider  
the {\it noisy polynomial  interpolation problem\/} modulo prime powers which is analogue to the same problem in finite fields~\cite{Shp1,ShpWint},
which in turn is an extension of the {\ hidden number problem} of Boneh and Venkatesan~\cite{BV1,BV2}.

To be more precise, for an integer $m$ we denore by $\Z_m = \Z/m\Z$ the residue ring modulo an integer $m \ge 1$,
and by $\Z_m^*$ the group of units of $\Z_m$.

Then the {\it noisy polynomial  interpolation problem\/} is the problem of
finding an unknown $s$-sparse polynomial
\begin{equation}
\label{eq:poly}
f(X) = \sum_{j=1}^s a_j X^{e_j} \in \Z_m[X],
\end{equation}
with monomials of degrees  $e_s> \ldots > e_1\ge1$ from approximations to
the  values of  $f(t)$ (treated as  integers from  the set  $\{0, 1, \ldots, m-1\}$) 
at polynomially many points
$t \in\Z_m$ selected uniformly at random.

Several problems of this type are related to the so-called
{\it hidden number problem\/} introduced by Boneh and 
Venkatesan~\cite{BV1,BV2}, which corresponds to a linear polynomial $f(X) = aX$
with unknown $a$, and have already been 
studied intensively due to their cryptographic relevance, see
the survey~\cite{Shp2}.  For sparse polynomials this problem 
has been studied in~\cite{Shp1,ShpWint},   for  some recent 
modifications motivated by cryptographic applications, see~\cite{G-MRTS}.

More precisely for integers $u$ and $m \ge 1$ we denote by  $\flm{s}$
the remainder of $s$ on division by $m$.

Furthermore.
for integers $s$, $m \ge 1$ and a real    $\ell \ge 0$ we denote by  $\MSB_{\ell,m}(s)$
any integer $u$ such that
\begin{equation}
\label{Lbits}
  \left |\flm{s}-u \right | \le m/2^{\ell+1}.
\end{equation}
Roughly speaking, $\MSB_{\ell,m}(z)$ gives $\ell$ most significant
bits of the remainder on division of $z$ by $m$.
However, this definition is more flexible and suits
better  our purposes. In particular we remark that $\ell$ in the
inequality~\eqref{Lbits} is not necessarily an integer.

The {\it sparse polynomial noisy interpolation problem\/} is the problem of
finding a  polynomial $f(X) \in \Z_m[X]$ of the form~\eqref{eq:poly} with {\it known\/}
exponents $e_1, \ldots, e_s$ and {\it unknown\/} coefficients $a_1, \ldots, a_s\in \Z_m$. 
from approximate values of  $\flm{f(t)}$ at polynomially many points
$t \in\Z_m$
selected uniformly at random.  
We remark that we always assume that the exponents    $e_s, \ldots ,  e_1$ are positive since if 
$\ell$ is not very small, it is impossible to distinguish between $f(X)$ and $f(X)+1$. 

Here we are interested in the setting where the modulus  $m=p^k$ is a large power of 
a fixed prime, for example $m = 2^k$, while previous works~\cite{BV1,BV2, G-MRTS, Shp1,Shp2, ShpWint},  
address the case when $m=p$ is a large prime. In the case of $m=p^k$, we use the ideas of~\cite{Shp1, ShpWint} 
combined with new number theoretic tools, coming from~\cite{Shp0}, and  give a  polynomial time algorithm provided 
that for each $t$ slightly  more than a half of the bits of 
$\MSB_{\ell,p^k}(f(t))$ are given. 

We note that algorithm itself is {\it deterministic\/}, and the only {\it randomness\/}  is in the choice of the   evaluation 
points $t$, while the consecutive computation is deterministic. 
 
\section{Our results}

We recall that the notations $U = O(V)$,  $U \ll V$  and $V \gg U$ are all
equivalent to the assertion that the inequality $|U|\le c|V|$ holds for some
constant $c>0$, which throughout the paper may   depend on    the real positive parameters $\eta$ and $\varepsilon$, 
the  integer $s\ge 1$ 
and  the prime $p$. 

It is also convenient to define $\log z$ as the binary logarithm of real $z > 0$.

We always assume that 
$$
n = \rf{k\log p} 
$$
is the bit length of the modulus $q = p^k$. 

Our result depends on the $p$-divisibility of the following  determinant, 
formed by binomial coefficients
$$
\Delta(e_1, \ldots, e_s)= \det 
\begin{pmatrix}  \binom{e_1}{1}, & \ldots, &\binom{e_s}{1} \\
           \ldots,          & \ldots, &  \ldots \\
          \binom{e_1}{s}, & \ldots, &\binom{e_s}{s}
\end{pmatrix} = \prod_{i=1}^s\frac{e_i}{i!} \prod_{1\le i < j\le s} (e_j - e_i). 
$$

Finally, for  an integer $a \ne 0$ we denote by $\ord_p a$ the $p$-adic order of $a$, that is, 
the largest integer  $\alpha$  with $p^{\alpha} \mid a$ and by 
$$
\|a\|_p = p^{-\ord_p a} 
$$
the $p$-adic valuation of $a$.

\begin{theorem}
\label{thm:PolAppr} 
Let $q= p^k$ be a sufficiently large $n$-bit 
power of a fixed prime $p$ and let $s \ge 1$ be a fixed  integer. 
Assume that for the integers $1\le e_1 < \ldots < e_s < q$ and real $\ell$ we have
\begin{equation}
\label{eq:Cond rho eta}
 \ell \ge (0.5 + \varepsilon)n
\mand  \|\Delta(e_1, \ldots, e_s)\|_p  \ge s 2^{-\ell(1/(s+1) -\varepsilon) } 
\end{equation}
 for some fixed $\varepsilon> 0$. 
Then there exists a  deterministic polynomial time algorithm $\cA$
such that for any polynomial $f(X)    \in \Z_{q}[X]$
of the form~\eqref{eq:poly}, given $2d$ integers  \begin{equation}
\label{eq:tiwi}
t_i  \mand  w_i = \MSB_{\ell, q}\(f(t_i) \), \qquad i=1,\ldots, d,
\end{equation}
 where 
$$ 
d = \rf{ 4(s+1)\varepsilon^{-1}}, 
$$
its output satisfies
$$
\Pr_{t_1, \ldots, t_d \in \Z_q^*} \bigl[
\cA\(t_1, \ldots, t_d; w_1, \ldots, w_d \)  = (a_1, \ldots,a_s)\bigr] \ge 1- 1/q
$$
if  $t_1, \ldots, t_d$ are
chosen uniformly and independently at
random from $\Z_q^*$.
\end{theorem}

Analysing the proof of Theorem~\ref{thm:PolAppr} one can easily see that the value of $d$ is
not optimised and can be improved at the cost of more tedious calculations.  

We note that it is natural to expect that any result of the type of Theorem~\ref{thm:PolAppr}  should depend 
on the $p$-divisibility of  the  determinant
$\Delta(e_1, \ldots, e_s)$,   which in turn measures the $p$-adic closeness of the exponents in
different monomials of $f$ and thus controls  $p$-adic independence of these monomials.

\section{Congruences with sparse  polynomials}
\label{sec:Cong Poly}

For a polynomial $F  \in \Z_{m}[X]$ in a residue ring modulo $m\ge 2$, and   integers $a$ and  $h$, 
we denote by $N_F(a,h;m)$ the number of solutions to the congruence
\begin{equation}
\label{eq:Cong Fxu}
F(x) \equiv u \pmod m, \qquad x \in \Z_{m}^*,  \ u \in \{a+1, \ldots, a+h\}.
\end{equation}

A  natural and powerful tool to estimate  $N_F(a,h;m)$ is given by 
bounds on exponential sums 
$$
S(F,m) =  \sum_{\substack{x = 1\\ \gcd(x,m) = 1}}^m  \e_m (F(x)), 
$$
where 
$$
\e_m(z) = \exp(2 \pi i z/m).
$$

In fact a bound on such sums  for a sparse polynomial $F$ as in~\eqref{eq:poly} has been given in~\cite[Theorem~1]{Shp0}, 
which however requires that the $\deg f$ is bounded (independently of $q$) and thus makes it is very restrictive for our applications. 
Bourgain~\cite{Bourg} has given  different versions of this result and relaxed 
the condition of $\deg f$ however  the corresponding bounds are weaker. 
 
Here we exploit the fact that   the results and method of~\cite{Shp0} allow us  to obtain a 
bound   on $N_F(a,h;q)$ which depends on the $p$-divisibility of $\Delta(e_1, \ldots, e_s)$ (which controls $p$-adic  properties of the differences between exponents
$e_1,\ldots, e_n$) rather than on $\deg f = \max\{e_1,\ldots, e_n\}$. 

First we need a slightly modified  and explicit version of a result from~\cite{Shp0}. 

\begin{lemma}
\label{lem:ExpSum}
Let  $q=p^\alpha$ be a power of a fixed prime $p$ and let 
$$
F(X)    =  \sum_{j=1}^s A_j X^{e_j} \in \Z_q[X]
$$
be  a   polynomial   such that 
$$
\gcd(A_1, \ldots, A_s, p)  = 1 \mand e_s > \ldots > e_1 \ge 1.
$$
If for some   fixed $\eta > 0$ we have 
$$
p^\rho \le s^{-1} q^\eta, 
$$
where 
$$
\rho = \ord_p \Delta(e_1, \ldots, e_s), 
$$
then  
$$
\left|S(F,q) \right | \ll   q^{1-1/s +   \eta(s+1)/s}.
$$
\end{lemma}

\begin{proof} We essentially follow the proof of~\cite[Lemma~5]{Shp0} and trace the dependence on $\rho$. 
In particular, as in~\cite{Shp0} we fix some $\eta>0$ and define $\beta = \fl{\alpha \eta} + 1$
and define the integer $m$ by the inequalities 
$$
\beta m < \alpha < \beta (m+1).
$$

Let us also  define the following differential operators 
$$
D_\nu = \frac{1}{\nu !} \cdot   \frac{d^\nu} { d x^\nu} 
\qquad \nu = 0, 1, \ldots .
$$
Finally, let 
$$
\vartheta_x = \max\{  \beta (\nu-1) +  \ord_p D_\nu F(x):~\nu =1, \ldots, s \}. 
$$
By~\cite[Lemma~2]{Shp0}, the inequality 
\begin{equation}
\label{eq:theta rho}
\vartheta_x \le \rho + \beta (s-1) 
\end{equation}
holds, provided that $\gcd(x,p)=1$. 

We now  note by our assumption on the size of $p^\rho$ we have 
$$
\beta s - \vartheta_x \ge \beta - \rho \ge \alpha \eta - \rho > \frac{\log s}{\log p}, 
$$
which ensures that the condition of~\cite[Lemma~4]{Shp0} is verified, see also~\cite[Equation~(5)]{Shp0} 

Then by~\cite[Equation~(6)]{Shp0} we have 
\begin{equation}
\label{eq:S sigma}
\left|S(F,q) \right |  \le  \sum_{\substack{x = 1\\ \gcd(x,p) = 1}}^{p^\beta} \left|  \sigma(x)\right |,
\end{equation}
where 
$$
\sigma(x) =    \sum_{ y=1}^{p^{\alpha -\beta}}  
 \e_{p^{\alpha -\beta}}\( \sum_{\nu = 1}^m  p^{\beta(\nu-1)}y^\nu D_\nu F(x)\). 
$$
We can certainly assume that $\alpha$ is large enough (in terms of $\eta$) and thus we are in the case $\alpha-\beta > 1$
of the proof of~\cite[Lemma~5]{Shp0}.  In this case, by~\cite[Equation~(8)]{Shp0}, we have 
$$
\sigma(x) = O\(p^{\alpha(1-1/s)}\)
$$
provided that $\rho$ is fixed. We now trace the dependence on $\rho$, which is explicit in the proof of~\cite[Lemma~5]{Shp0} 
till the very last step. 

More precisely, it is shown in the proof of~\cite[Lemma~5]{Shp0}  that for $\vartheta_x  \ge \alpha -\beta$ we have 
$$
 \left|  \sigma(x)\right |  \le p^{\rho + \beta (s-1)}. 
 $$
Hence we see from~\eqref{eq:S sigma} that the total contribution to the bound on $\left|S(F,q) \right |$ 
from all such values of $x$,  which we denote by $\Sigma_1$, is at most 
\begin{equation}
\label{eq:S1}
\Sigma_1 \le p^\beta  p^{\rho + \beta (s-1)} =  p^{\rho + \beta s}\le   p^{\rho + \eta \alpha s + s}. 
\end{equation}

Furthermore, in the case when  $\vartheta_x  < \alpha -\beta$ it is shown in the proof of~\cite[Lemma~5]{Shp0}  that 
for some constant $c$ which depends only on $t$ and $m$ (and thus only on  $t$ and $\eta$)
$$
 \left|  \sigma(x)\right |  \le c p^{\vartheta_x + (\alpha -\beta - \vartheta_x) (1-1/s)}
$$
and also that the exponent satisfies the inequality 
\begin{align*}
  \vartheta_x + (\alpha -\beta - \vartheta_x) (1-1/s) & =   \(\vartheta_x - \beta (s-1)\)/s + \alpha(1-1/s)\\
  &\le \rho/s + \alpha(1-1/s),
\end{align*}
which follows from~\eqref{eq:theta rho}. 
  Hence we see from~\eqref{eq:S sigma} that the total contribution to the bound on $\left|S(F,q) \right |$ 
from all such values of $x$, which we denote by $\Sigma_2$,  is at most 
\begin{equation}
\label{eq:S2}
\Sigma_2 \le  c p^\beta  p^{\rho/s + \alpha(1-1/s)} \le c  p^{\rho/s +   \alpha(1-1/s) +  \eta \alpha}.
\end{equation}

Combining~\eqref{eq:S1} and~\eqref{eq:S2}, we see  from~\eqref{eq:S sigma}  and the assumed rrestriction of $p^\rho$ that 
\begin{align*}
\left|S(F,q) \right | &  \le \Sigma_1 + \Sigma_2  \le    p^{\rho + \eta \alpha s + s} + c  p^{\rho/s +   \alpha(1-1/s) +  \eta \alpha}\\
& \ll  p^{\eta \alpha (s +1)} +   p^{\eta \alpha/s +\alpha(1-1/s) +  \eta \alpha}
\ll  q^{\eta  (s +1)} +   q^{1-1/s +   \eta(s+1)/s}.
\end{align*}
We can assume that $\eta \le 1/(s+1)$ as otherwise the second term in the above inequality exceeds the
trivial bound $q$. On the other hand for  $\eta< 1/(s+1)$ we have 
$$
\eta  (s +1)< 1-1/s +   \eta(s+1)/s.
$$
Hence the second term always 
dominates and the desired bound follows. 
 \end{proof}
 
 We remark that the bound of exponential sums of~\cite[Lemma~5]{Shp0}, which underlies the proof of Lemma~\ref{lem:ExpSum},  is  based, in turn 
 or a result of Mit'kin~\cite[Lemma~1.2]{Mit}. Using a bound of Cochrane and  Zheng~\cite[Equation~(2.11)]{CoZhe}
 one can get better values of implied constants; this  however does not affect our main result.

Combining Lemma~\ref{lem:ExpSum} with  the classical {\it  Erd\H{o}s--Tur\'{a}n inequality\/} (see, for example,~\cite[Theorem~1.21]{DrTi}), 
which links the irregularity of distribution of sequences to exponential sums, we immediately derive that 
$N_F(a,h;q) $ is close to its expected value
$$
 h\frac{\varphi(q)}{q} =  h\frac{p-1}{p}, 
$$
where $\varphi(q)$ is the Euler function. More precisely, we recall that  the {\it discrepancy\/} $D(N) $ of a sequence in $\xi_1, \ldots, \xi_N \in [0,1)$    is defined as 
$$
D(N) = 
\sup_{0\le \alpha < \beta \le 1} \left |  \#\{1\le n\le N:~\xi_n\in [\alpha, \beta)\} -(\beta-\alpha)  N \right |, 
$$
where  $\# \cS$ denotes  the cardinality of  $\cS$ (if it is finite), see~\cite{DrTi} for background.

By  the classical Erd\H{o}s--Tur\'{a}n inequality (see, for instance,~\cite[Theorem~1.21]{DrTi}) we have the following 
estimate of the discrepancy via exponential sums. 

\begin{lemma}
\label{lem:ET}
Let $\xi_n$, $n\in \N$,  be a sequence in $[0,1)$. Then for any $H\in \N$,  we have 
$$
D(N)  \le 3 \left( \frac{N}{H+1} +\sum_{h=1}^{H}\frac{1}{h} \left| \sum_{n=1}^{N} \exp(2\pi i h\xi_n) \right | \right).
$$
\end{lemma}

We now interpret the congruence~\eqref{eq:Cong Fxu} as a condition on the fractional parts $\{F(x)/q\}$ to fall in a certain interval 
of a unit torus $\R/\Z \cong [0,1)$ 
of length $h/q$. This immediately implies the desired result.

\begin{lemma}
\label{lem:Congr}
Let  $q=p^\alpha$ be a power of a fixed prime $p$ and let 
$$
F(X)    =  \sum_{j=1}^s A_j X^{e_j}\in \Z_q[X]
$$
be  a   polynomial   such that 
$$
\gcd(A_1, \ldots, A_s, p)  = 1 \mand e_s > \ldots > e_1 \ge 1.
$$
If for some  fixed $\varepsilon > 0$ we have 
$$
p^\rho \le  s^{-1}   q^{1/(s+1) -\varepsilon},
$$
where 
$$
\rho = \ord_p \Delta(e_1, \ldots, e_s), 
$$
then   
$$
\left| N_F(a,h;q)  - h\frac{p-1}{p} \right| \ll  q^{1-\varepsilon} \log q  .
$$
\end{lemma}

\begin{proof} We set 
$$
\eta = \frac{1-(s+1)\varepsilon}{\(s+1\)\(1-\varepsilon\)} \mand
H =   \fl{q^{\varepsilon}}.
$$ 
For each exponential sum which appears in the bound of Lemma~\ref{lem:ET}  corresponding to $h$ with $\gcd(h,q)=d$, we have 
\begin{equation}
\label{eq:Red Sum}
 \sum_{\substack{x = 1\\ \gcd(x,p) = 1}}^q  \e_q (hF(x))
 = d \sum_{\substack{x = 1\\ \gcd(x,p) = 1}}^q  \e_{q/d} ((h/d)F(x))
\end{equation}
Furthermore, for each $h=1, \ldots, H$ we note that 
\begin{equation}
\label{eq:Cond rho}
s^{-1}  \(q/d\)^{\eta}\ge s^{-1}  \(q/H\)^{\eta} \ge   s^{-1}   \(q^{1-\varepsilon}\)^{\eta} 
=s^{-1}   q^{1/(s+1) -\varepsilon}
 \ge p^\rho . 
\end{equation}
We also observe that 
\begin{equation}
\begin{split} 
\label{eq:exponent}
1/s  - \eta(s+1)/s= \frac{1}{s} \(1 -\eta(s+1) \) 
= \frac{1}{s}\(1 -   \frac{1-(s+1)\varepsilon}{1-\varepsilon} \) = 
 \frac{\varepsilon}{1-\varepsilon}. 
\end{split} 
\end{equation}
We see from~\eqref{eq:Cond rho}  that we can now apply Lemma~\ref{lem:ExpSum} to the sum on the right hand side of~\eqref{eq:Red Sum} 
with the above $\eta$. Thus we taking into account our calculation in~\eqref{eq:exponent}, we obtain 
\begin{align*}
 \sum_{\substack{x = 1\\ \gcd(x,p) = 1}}^q  \e_q (hF(x))& \ll 
 d (q/d)^{1-1/s +    \eta(s+1)/s}\
 = q (q/d)^{-1/s +    \eta(s+1)/s}\\
 &  = q (q/d)^{-\varepsilon/(1-\varepsilon)} \le    q \(q^{1-\varepsilon}\)^{-\varepsilon/(1-\varepsilon)}= q^{1-\varepsilon}, 
\end{align*}
Therefore,  Lemma~\ref{lem:ET} now yields
$$
\left| N_F(a,h;q)  - h\frac{p-1}{p} \right| \ll q^{1-\varepsilon} \log q .
$$
which concludes the proof. 
\end{proof} 


Finally, it is convenient to have an upper bound on $N_F(a,h;q) $ for polynomials with non-necessary co-prime with $p$ coefficients. 
Namely if for $F$ as in Lemma~\ref{lem:Congr} we have 
$$
\gcd(A_1, \ldots, A_s, q)  = D
$$
then provided $h \le q/D$ we have
\begin{equation}
\label{eq:Red}
N_F(a,h;q) = D N_{D^{-1} F}(a,h;q/D).
\end{equation}

\begin{cor}
\label{cor:CongGen} 
Let  $q=p^\alpha$ be a power of a fixed prime $p$ and let 
$$
F(X)    =  \sum_{j=1}^s A_j X^{e_j} \in \Z_q[X]
$$
be  a   polynomial   such that 
$$
\gcd(A_1, \ldots, A_s, q)  = D  \mand e_s > \ldots > e_1 \ge 1.
$$
If for some  fixed $\varepsilon > 0$ we have 
$$
p^\rho \le  s^{-1}   (q/D)^{1/(s+1) -\varepsilon},
$$
where 
$$
\rho = \ord_p \Delta(e_1, \ldots, e_s), 
$$
then 
$$
  N_F(a,h;q) \le D h  + D^{\varepsilon} q^{1-\varepsilon} \log q.
$$
\end{cor}

Note that in Corollary~\ref{cor:CongGen} we have abandoned the condition $h \le q/D$
which is needed for~\eqref{eq:Red} as for $h >q/D$ its bound is trivial.

\section{Background on lattices}
\label{sec:Lat CVP}

As in~\cite{BV1,BV2}, and then in~\cite{Shp1,ShpWint}, 
our results rely on 
some lattice algorithms. 
 We therefore review some relevant results and definitions, we refer 
to~\cite{ConSlo,GrLoSch,GrLe} for more details and the general theory.

Let $\{{\vb}_1,\ldots,{\vb}_N\}$ be a set of $N$ linearly independent vectors in
${\R}^N$. The set
of  vectors
$$
\cL=\left \{\vz  \ : \ \vz=\sum_{i=1}^N\ c_i\vb_i,\quad c_1, \ldots, c_N\in\Z\right \}
$$
is called an $s$-dimensional full rank lattice. 

The set $\{ {\vb}_1,\ldots, {\vb}_N\}$ is called a {\it basis\/} of $\cL$.

The volume of the parallelepiped defined by the 
vectors ${\vb}_1,\ldots, {\vb}_N$ is called {\it 
the volume\/} of the lattice and denoted by $\vol(\cL)$. 
Typically, lattice problems are easier when the Euclidean 
norms of all basis vectors are close to $\vol(\cL)^{1/N}$. 

Let $\|\vz\|$ denote the standard Euclidean norm in ${\R}^N$.

One of the most fundamental problems in this area is 
 the {\it closest vector problem\/},  CVP:
given a basis of a lattice $\cL$ in $\R^N$ and a target vector $\vec{u} \in \R^N$,
find a lattice vector $\vec{v} \in \cL$ which minimizes
the Euclidean norm $\|\vec{u}-\vec{v}\|$ among all lattice vectors.
It is well know that  CVP is  {\bf NP}-hard when the 
dimension $N\to \infty$
(see~\cite{MicVou,Ngu,NgSt1,NgSt2,Reg} for references). 


There are several approximate algorithms to find vectors in lattices which are close to a given target vector $\vr =(r_1, \ldots, r_N) \in{\R}^N$, 
see~\cite{AKS,Kan,Schn} which build on the    classical lattice basis reduction algorithm of Lenstra, Lenstra
and  Lov{\'a}sz~\cite{LLL}, we also refer to~\cite{Ngu,NgSt1,NgSt2,Reg} for possible improvements and further references.

However, it is important to observe that  in our case, the dimension of the lattice is bounded so we can use 
one of the deterministic  algorithms which finds the closest vector exactly. 
For example, we appeal to  the following result of 
Micciancio and  Voulgaris~\cite[Corollary~5.6]{MicVou}.

\begin{lemma}
\label{lem:CVP} Assume that we are given a basis of a lattice $\cL$, which consists of 
vectors of rational numbers ${\vb}_1,\ldots,{\vb}_N\in \Q^N$ and  a vector $\vr  \in{\Q}^N$
such that their numerators and denominators of ${\vb}_1,\ldots,{\vb}_N, \vr$ are at most $n$-bits long. 
There is  a deterministic algorithm which for  a fixed $N$, in time  polynomial in $n$,  finds a lattice vector
$\vv = (v_1, \ldots, v_N) \in \cL$ satisfying the inequality
$$
  \|\vv - \vr\|  =  \min \left\{  \|\vv - \vz\|~:~ \vz   \in \cL\right\}.
$$
\end{lemma}
%

\section{Lattices and polynomial approximations}
\label{sec:Lap Poly}

Let $\ve =    (e_1, \ldots,  e_s)$.

For  $t_1, \ldots, t_d \in \Z_q$, we denote by $ \cL_{\ve,q}\(t_1, \ldots, t_d\)$ the $(d+s)$-dimensi\-onal lattice
generated by the rows of the following $(d +s)\times (d+s)$-matrix
\begin{equation}
\label{eq:matrix}
 \begin{pmatrix}
 q & 0 & \ldots &  0   & 0&  \ldots  & 0\\
 0 & q & \ldots &  0   &0&  \ldots  & 0\\
 \vdots & {} &\ddots& {} &\vdots & { } & \vdots\\
 0 & 0 &  \ldots & q &  0 & \ldots & 0  \\
 t_1^{e_1} & t_2^{e_1} &  \ldots & t_d^{e_1} &  1/2^{n+1}  & \ldots & 0 \\
 \vdots & {} & { } & {} & \vdots &\ddots & \vdots\\
 t_1^{e_s} & t_2^{e_s} & \ldots & t_d^{e_s} &   0& \ldots & 1/2^{n+1}
 \end{pmatrix}. 
\end{equation}

The following result is a  generalization of
several previous results of similar flavour obtained for a large prime number $q=p$, see~\cite{BV1,Shp1,ShpWint}. 

\begin{lemma}
\label{lem:Latt-Poly}
Let $q= p^k$ be a sufficiently large $n$-bit 
power of $p$ and let $s \ge 1$ be a fixed  integer. 
Let  $f(X)    \in \Z_{q}[X]$
be  a   polynomial  
of the form~\eqref{eq:poly} 
with known  exponents $1\le e_1 < \ldots < e_s < q$. 
If  the conditions~\eqref{eq:Cond rho eta} hold then  for 
$$d = \rf{ 2(s+1) \varepsilon^{-1}},
$$ 
t the following holds. If  $t_1, \ldots, t_d \in \Z_q^*$ are
chosen uniformly and independently at random, then with
 probability $P \ge 1 - 1/q$ 
for any vector $\vu= (u_1, \ldots, u_d, 0, \ldots, 0)$ with
$$
\(\sum_{i=1}^{d} \(\flq{f(t_i)} - u_i\)^2\)^{1/2} \le  2^{-\ell} q,
$$
all vectors
$$
\vv = (v_1, \ldots, v_{d}, v_{d+1}, \ldots, v_{d+s}) 
  \in  \cL_{\ve,q}\(t_1, \ldots, t_d\)
$$
satisfying
$$
\(\sum_{i=1}^d \(v_i - u_i\)^2\)^{1/2} \le  2^{-\ell} q,
$$
are of the form
$$
\vv =\( \flq{\sum_{j=1}^s b_j t_1^{e_j}}, \ldots,  
\flq{\sum_{j=1}^s b_j t_d^{e_j} },   b_1/2^{k+1}, \ldots, b_s/2^{k+1}\)
$$
with some integers $b_j \equiv a_j \pmod q$, $j =1, \ldots, s$.
\end{lemma}

\begin{proof} As in~\cite{BV1} we define the modular distance between two
integers $\beta$ and $\gamma$ as
\begin{align*}
\dist_q(\beta,\gamma) &  =   \min_{b \in \Z} |\beta - \gamma - bq|\\
&  = 
\min \left\{\flq{\beta - \gamma}\,,
q - \flq{\beta - \gamma} \right\}.
\end{align*}

Let $\cP_f$ denote the set of $q^s-1$ polynomials  
$$
g(X) = \sum_{j=1}^s b_j X^{e_j} \in \Z_{q}[X]
$$
with $g \ne f$.

For a polynomial $g \in \cP_f$ we denote by $P(g)$
the probability  that 
\begin{equation}
\label{eq:Close f g}
\dist_q( f(t), g(t))  \le  2^{-\ell+ 1} q,
\end{equation}
for $t \in \Z_{q}^*$ selected uniformly at random. 
To estimate $P(g)$ we consider the polynomial 
\begin{equation}
\label{eq:poly F}
F(X) = f(X)-g(X) =   \sum_{j=1}^s A_j X^{e_j} \in \Z_q[X].
\end{equation}
Clearly, $F(X)$ is not identical to zero in $\Z_q$. 
Hence, if~\eqref{eq:Close f g} is possible for some $t \in \Z_q^*$, then for 
$$
D = \gcd(A_1, \ldots, A_s, q) 
$$
we have 
\begin{equation}
\label{eq:D}
D \le   2^{-\ell+ 1} q.
\end{equation}
Therefore, by our assumption, we have 
$$
 s^{-1}   (q/D)^{1/(s+1) -\varepsilon} \ge  s^{-1} 2^{\ell(1/(s+1) -\varepsilon) } \ge p^\rho, 
 $$
 and thus  Corollary~\ref{cor:CongGen} applies. 

We now set 
\begin{equation}
\label{eq:a h}
a = -  \fl{ 2^{-\ell+ 1} q} \mand h = 2 \rf{ 2^{-\ell+ 1} q} + 1. 
\end{equation}
We see from Corollary~\ref{cor:CongGen} that
$$
P(g) =    \frac{1}{\varphi(q)} N_F(a,h;q) \le  \frac{1}{\varphi(q)} \(D h  +   D^{\varepsilon} q^{1-\varepsilon} \log q\).
$$
Hence, recalling the bound~\eqref{eq:D} and the choice of $h$ in~\eqref{eq:a h} we obtain
$$
P(g) \ll   2^{-2\ell} q +   2^{-\ell \varepsilon }   \log q.
$$
Recalling the inequalities~\eqref{eq:Cond rho eta}, we obtain
$$
2^{-2\ell} q \ll 2^{n -2\ell}  \ll 2^{-2\varepsilon n} \ll q^{-2\varepsilon}
$$
and 
$$
2^{-\ell\varepsilon  }  \log q \ll q^{(1/2+\varepsilon)\varepsilon)} \log q \le   q^{- \varepsilon/2}.
$$
Hence  
$$
P(g) \ll   q^{-  \varepsilon/2}  
$$
provided that $q$ is large enough.

Therefore, for any  $g \in \cP_f$,
\begin{align*}
 \Pr&\left[\exists i\in [1,d] \ | \ \dist_q( g(t_i), f(t_i)) >   2^{-\ell+ 1} p\right]\\
 &  \qquad \qquad \qquad \qquad  \qquad   = 1 - P(g)^d\ge 1 - q^{-d  \varepsilon/2},
\end{align*}
where the probability is taken over $t_1, \ldots, t_d \in \Z_{q}$
chosen uniformly and independently at random.

Since $\# \cP_f = q^s -1$,  taking
$$
d = \rf{ 2(s+1) \varepsilon^{-1}}, 
$$ 
we obtain 
\begin{align*}
\Pr\left[\forall  g \in \cP_f, \ 
\exists i\in [1,d] \ |  \ \dist_q(g(t_i), f(t_i)) >  2^{-\ell +1} q\right] & \\
\ge 1 - (q^s - 1) q^{-d  \varepsilon/2}& > 1 - 1/q, 
\end{align*}
provided that $q$ is large enough.

The rest of the
proof is essentially  identical to the proof of~\cite[Theorem~5]{BV1},
see also  the proof of~\cite[Theorem~8]{Shp1}. 
Indeed, we fix some integers $t_1, \ldots, t_d$ with
\begin{equation}
\label{Condition}
\min_{g \in \cP_f} \  \max_{i\in [1,d]}
 \dist_q(g(t_i), f(t_i))  >  2^{-\ell +1} q.
\end{equation}

Let $\vv \in  \cL_{\ve,q}\(t_1, \ldots, t_d\)$ be a lattice point satisfying
$$
\(\sum_{i=1}^d \(v_i - u_i\)^2\)^{1/2} \le  2^{-\ell} q.
$$
Clearly, since $\vv\in  \cL_{\ve,q}\(t_1, \ldots, t_d\)$,
there are some  integers $\beta_1, \ldots, \beta_s$ and $z_1,\ldots,z_d$ such that
$$
\vv   =   \(\sum_{j=1}^s \beta_j t_1^{e_j} - z_1 q,\ldots,\sum_{j=1}^s \beta_j t_d^{e_j} - z_d
q, \beta_1/2^{k+1}, \ldots, \beta_s/2^{k+1}\).
$$

If  $\beta_j \equiv \alpha_j \pmod q$, $j =1, \ldots, s$, then for all $i=1,\ldots,d$ we have
$$
\sum_{j=1}^s \beta_j t_i^{e_j} - z_i q = \flq{\sum_{j=1}^s \beta_j t_i^{e_j}} = \flq{f(t_i)},
$$
since otherwise there
is 
$i\in [1,d] $ such that 
$|v_i - u_i | >  2^{-\ell} q$.

Now suppose that $\beta_j \not \equiv \alpha_j \pmod q$ for some $j \in [1,s]$. 
In this case we have
\begin{align*}
  \(\sum_{i=1}^d \(v_i - u_i\)^2\)^{1/2} &  \ge  \min_{i \in [1,d]}
 \dist_q\(\sum_{j=1}^s \beta_j t_i^{e_j}, u_i\) \\ 
&   \ge   \min_{i \in [1,d]}  
 \dist_q\(f(t_i), \sum_{j=1}^s \beta_j t_i^{e_j}\)  -  \dist_q\(u_i, f(t_i)\)\\
&     >   2^{-\ell +1} q -  2^{-\ell} q =  2^{-\ell} q, 
\end{align*}
that
contradicts  our assumption.
As we have seen, the condition~(\ref{Condition}) holds with probability
exceeding $1- 2^{-\ell}$ and the result follows.
\end{proof}

\section{Proof of Theorem~\ref{thm:PolAppr}}

As in all previous works, we follow the same arguments as in the proof of 
of~\cite[Theorem~1]{BV1}
which we  briefly outline here for the sake of completeness. We refer to
the first $d$ vectors in  the
matrix~\eqref{eq:matrix} as $q$-vectors and we refer to the other $s$ vectors 
as power-vectors.

We recall~\eqref{eq:tiwi} and consider the vector 
$$\vw =(w_1, \ldots, w_d,w_{d+1}, \ldots, w_{d+s})
$$
where
$$
w_{d+j} =0, \qquad j =1, \ldots, s.
$$

We can certainly assume that 
$$
\ell \le n
$$ 
as otherwise the result is trivial. 
Then multiplying the  $j$th power-vector  of the
matrix~\eqref{eq:matrix} 
 by $\alpha_j$ and subtracting a
 certain multiple of the $j$th $q$-vector, $j =1, \ldots, s$, we obtain a lattice point
$$
{\vu}_{f} =  (u_1,\ldots,u_d,\alpha_1/2^{n+1}, \ldots, \alpha_s/2^{n+1})\  \in   \cL_{\ve,q}\(t_1, \ldots,
t_d\)
$$
 such that
$$
\left|u_i - w_i\right| < q2^{-\ell-1}, \qquad i = 1, \ldots, d+s, 
$$
where $u_{d+j} = \alpha_j/2^{n+1}$, $j =1, \ldots, s$.
Therefore,
$$
\sum_{i=1}^{d+s} \(u_i - w_i\)^2  \le  (d+s)  2^{-2\ell -2} q^2.
$$
We can assume that $q$ is large enough so that $s \le n$. 
Therefore $d+s = O(n)$. 
Now we can use Lemma~\ref{lem:CVP}  to find in polynomial
time a lattice vector
$\vv = (v_1,\ldots,v_d,v_{d+1},  \ldots, v_{d+s} )\in  \cL_{\ve,q}\(t_1, \ldots, t_d\)$ such that
\begin{align*}
 \sum_{i=1}^d \(v_i - w_i\)^2 
& =  \min
\left\{\sum_{i=1}^{d+s} \(z_i - w_i\)^2 ,\quad \vz = \(z_1,
\ldots, z_{d+s}\) \right\}\\
& \le  (d+s)  2^{-2\ell -2} q^2 \le  2^{-2\ell_0 -2} q^2,
 \end{align*}
provided that $q$ is large enough, where, for example, we can choose
$$
\ell_0  =  (0.5 + \varepsilon/2)n. 
$$  
Applying Lemma~\ref{lem:Latt-Poly} with $\ell_0$ in place of $\ell$, and thus with $\varepsilon/2$
in place of $\varepsilon$ we see that $\vv={\vu}_{f}$ with probability
at least $1- 1/q$, and therefore the coefficients of  $f$ can be recovered in
polynomial time.

\section{Comments}

It seems like a natural idea to classify polynomials  $g \in \cP_f$  in the proof of Theorem~\ref{thm:PolAppr} 
depending on the size of $D = \gcd(A_1, \ldots, A_s, q)$ where $A_1, \ldots, A_s$ are as in~\eqref{eq:poly F}, 
instead of using  the worst case  bound~\eqref{eq:D} 
We can then  take into account that for a given $D = p^r$ there are at most $(q/p^r)^s$ polynomials 
$g \in \cP_f$  with this values of $D$. Unfortunately, this approach may only help to reduce slightly 
the value of $d$ in Theorem~\ref{thm:PolAppr}, which is not optimised anyway. 

We remark that here we essentially consider the interpolation  problem when the values of a polynomial $f$ are corrupted 
by an additive noise. That is, for any $t \in \Z_q^*$ we are given $f(t) + \vartheta$ for some $\vartheta\in  \Z_q$ which is 
not too large.   For a large prime $q = p$, in~\cite{vzGSh}  the case of multiplicative noise  has been studied, 
where  for any $t \in \Z_q^*$ we are given the residue 
modulo $q$ of $\rho f(t)$ 
for some rational  $\rho$ with not too large numerator and denominator.  It is certainly   interesting to consider this scenario
with multiplicative noise modulo  powers  of small primes as in this work. 

 \section*{Acknowledgement}
 
This work started  during a very enjoyable visit of the second author 
to the University of Bonn, whose hospitality is very much appreciated. 
This visit was supported by  the excellence grant EXC~2-1 of the
Hausdorff Center for Mathematics. 

During the preparation of this work the first author was supported in part by the Deutsche Forschungsgemeinschaft 
and the second author by the  Australian Research Council.

\end{document}